\documentclass[10pt]{article}

\usepackage{amsmath, amsthm, amssymb, enumerate}
\usepackage{color}
\usepackage{datetime}
\usepackage{fancyhdr}
\usepackage{mathtools}

\chardef\coloryes=0 
\chardef\isitdraft=0 

\ifnum\isitdraft=1 
   \def\eqref#1{({\ref{#1}})}                

\else

  \def\startnewsection#1#2{\section{#1}\label{#2}\setcounter{equation}{0}}   
  \def\nnewpage{} 
\fi

\newcommand{\Unormlong}{\|U\|_{L^2 ({\mathbb{R}^n})}}
\newcommand{\Unorm}{\|U\|}
\newcommand{\Lnorm}{L^2 (\mathbb{R}^n)}
\DeclarePairedDelimiterX{\norm}[1]{\lVert}{\rVert}{#1}

\begin{document}
\def\ques{{\colr \underline{??????}\colb}}
\def\nto#1{{\colC \footnote{\em \colC #1}}}
\def\fractext#1#2{{#1}/{#2}}
\def\fracsm#1#2{{\textstyle{\frac{#1}{#2}}}}   
\def\nnonumber{}


\def\colr{{}}
\def\colg{{}}
\def\colb{{}}
\def\cole{{}}
\def\colA{{}}
\def\colB{{}}
\def\colC{{}}
\def\colD{{}}
\def\colE{{}}
\def\colF{{}}

\ifnum\coloryes=1

  \definecolor{coloraaaa}{rgb}{0.1,0.2,0.8}
  \definecolor{colorbbbb}{rgb}{0.1,0.7,0.1}
  \definecolor{colorcccc}{rgb}{0.8,0.3,0.9}
  \definecolor{colordddd}{rgb}{0.0,.5,0.0}
  \definecolor{coloreeee}{rgb}{0.8,0.3,0.9}
  \definecolor{colorffff}{rgb}{0.8,0.3,0.9}
  \definecolor{colorgggg}{rgb}{0.5,0.0,0.4}

 \def\colg{\color{colordddd}}
 \def\colb{\color{black}}
 \def\colr{\color{red}}
 \def\cole{\color{colorgggg}}

 \def\colA{\color{coloraaaa}}
 \def\colB{\color{colorbbbb}}
 \def\colC{\color{colorcccc}}
 \def\colD{\color{colordddd}}
 \def\colE{\color{coloreeee}}
 \def\colF{\color{colorffff}}
 \def\colG{\color{colorgggg}}

\fi
\ifnum\isitdraft=1
   \chardef\coloryes=1 
   \baselineskip=17pt
   \input macros.tex
   \def\blackdot{{\color{red}{\hskip-.0truecm\rule[-1mm]{4mm}{4mm}\hskip.2truecm}}\hskip-.3truecm}
   \def\bdot{{\colC {\hskip-.0truecm\rule[-1mm]{4mm}{4mm}\hskip.2truecm}}\hskip-.3truecm}
   \def\purpledot{{\colA{\rule[0mm]{4mm}{4mm}}\colb}}
   \def\pdot{\purpledot}
\else
   \baselineskip=15pt
   \def\blackdot{{\color{red}{\hskip-.0truecm\rule[-1mm]{4mm}{4mm}\hskip.2truecm}}\hskip-.3truecm}
   \def\pdot{{\color{blue}{\hskip-.0truecm\rule[-1mm]{4mm}{4mm}\hskip.2truecm}}\hskip-.3truecm}
\fi

\definecolor{labelkey}{gray}{.1}
\definecolor{refkey}{gray}{.1}

\def\tdot{\fbox{\fbox{\bf\tiny I'm here; \today \ \currenttime}}}
\def\nts#1{{\hbox{\bf ~#1~}}} 
\def\nts#1{{\colr\hbox{\bf ~#1~}}} 
\def\ntsf#1{\footnote{\hbox{\bf ~#1~}}} 
\def\ntsf#1{\footnote{\colr\hbox{\bf ~#1~}}} 
\def\bigline#1{~\\\hskip2truecm~~~~{#1}{#1}{#1}{#1}{#1}{#1}{#1}{#1}{#1}{#1}{#1}{#1}{#1}{#1}{#1}{#1}{#1}{#1}{#1}{#1}{#1}\\}
\def\biglineb{\bigline{$\downarrow\,$ $\downarrow\,$}}
\def\biglinem{\bigline{---}}
\def\biglinee{\bigline{$\uparrow\,$ $\uparrow\,$}}

\def\tilde{\widetilde}

\newtheorem{Theorem}{Theorem}[section]
\newtheorem{Corollary}[Theorem]{Corollary}
\newtheorem{Proposition}[Theorem]{Proposition}
\newtheorem{Lemma}[Theorem]{Lemma}
\newtheorem{Remark}[Theorem]{Remark}
\newtheorem{definition}{Definition}[section]
\def\theequation{\thesection.\arabic{equation}}
\def\endproof{\hfill$\Box$\\}
\def\square{\hfill$\Box$\\}
\def\comma{ {\rm ,\qquad{}} }            
\def\commaone{ {\rm ,\qquad{}} }         
\def\dist{\mathop{\rm dist}\nolimits}    
\def\sp{\mathop{\rm sp}\nolimits}    
\def\sgn{\mathop{\rm sgn\,}\nolimits}    
\def\Tr{\mathop{\rm Tr}\nolimits}    
\def\div{\mathop{\rm div}\nolimits}    
\def\supp{\mathop{\rm supp}\nolimits}    
\def\divtwo{\mathop{{\rm div}_2\,}\nolimits}    
\def\re{\mathop{\rm {\mathbb R}e}\nolimits}    

\def\indeq{\qquad{}}                     
\def\period{.}                           
\def\semicolon{\,;}                      

\title{Quantitative unique continuation for a parabolic equation}
\author{Guher Camliyurt and Igor Kukavica}
\maketitle

\date{}


\medskip

\indent Department of Mathematics, University of Southern California, Los Angeles, CA 90089\\
\indent e-mails: camliyur@usc.edu, kukavica\char'100usc.edu

\begin{abstract}
We address the quantitative uniqueness properties of the 
solutions of the parabolic equation
$
   \partial_t u - \Delta u 
   = w_j (x,t) \partial_j u 
     + v(x,t) u 
$ where $v$ and $w$ are bounded. 
We prove that for solutions $u$, the order of vanishing is bounded
by 
$C(\Vert v\Vert_{L^\infty}^{2/3}+\Vert w\Vert_{L^\infty}^2)$
matching the upper bound previously established in the elliptic case.
\end{abstract}



\startnewsection{Introduction}{sec1}  
The main purpose of this paper is to
provide quantitative uniqueness estimates
for a parabolic equation
  \begin{align}
   \partial_t u - \Delta u 
   = w_j (x,t) \partial_j u 
     + v(x,t) u 
   \label{EQ50}
  \end{align}
by 
finding an upper bound for the order of vanishing
of a solution in terms of the size of coefficients
$v$ and $w$.

Starting with the works of Carleman
\cite{C}, there have been many results showing that
solutions of the equations of parabolic or elliptic type have 
a finite order of vanishing, in which case we say that the equation
has the unique continuation property.
For instance, a result
of Escauriaza and Vega \cite{EV}
shows that for $v$ in a sharp space $L_t^{\infty}L_x^{n/2}$
(with the natural smallness assumption for the corresponding norm)
 and $w=0$, the equation \eqref{EQ50} has the unique continuation property.
Also, 
Koch and Tataru proved in \cite{KT}  that the equation 
\eqref{EQ11} has this property
provided $v$ and $w$ belong to appropriate Lebesgue spaces.
A sharp result of this type for the elliptic counterpart
was established by Jerison and Kenig 
\cite{JK} (cf.~\cite{K1,K2,K3,V} for reviews on unique continuation for
elliptic, parabolic, and dispersive equations).

More recently, considerable efforts were dedicated to
the quantitative estimates of unique continuation,
i.e., to estimating the maximal degree to which 
the solution may vanish at a point.
There are many applications of such results; in particular, 
quantitative uniqueness results
yield lower bounds on solutions of the corresponding PDE.
Also, estimates on the order of vanishing
are an essential tool for obtaining upper bounds
on the size of level sets of PDE \cite{DF1,DF2,DF3,Ku2,Ku3,Ku4,L},
in inverse problems and control \cite{AE},
upper bounds on vortex degrees \cite{Ku2},
spectral information on Schrodinger operators \cite{BK},
backward uniqueness \cite{EF},
Hadamard type theorems \cite{AMRV,B1,B2,BC,CRV,EFV,EVe,LNW},
and other topics.

The research on quantitative uniqueness was initiated by
Donnelly and Fefferman \cite{DF1,DF2,DF3},
who proved that the 
maximal order of vanishing ${\mathcal O}_{u}$
of an eigenfunction $u$ of an elliptic operator
is bounded by $C \sqrt \lambda$, where $\lambda$
is the corresponding eigenvalue. 
They used this estimate to prove that the nodal volume of a zero
set of an analytic eigenfunction is bounded by
$C\sqrt \lambda$.
It is
natural to ask the following question: Is the order of vanishing
of a solution of a boundary value problem
$-\Delta u=w \nabla u + v u$ subject to periodic, Dirichlet, or
Neumann conditions (or with prescribed constant doubling property)
bounded by
  \begin{equation}
    C\Vert v\Vert_{L^\infty}^{1/2}
      +  C \Vert w\Vert_{L^\infty}
    \label{EQ51}
  \end{equation}
\cite{Ku1}.
This question remains open and a precise upper bound
is currently unknown.
Based on an example by Meshkov
\cite{M}, who
provided a complex valued solution decaying with a certain
exponential rate at infinity, Bourgain and Kenig predicted that
(when  $w=0$), the optimal vanishing
rate is more likely
$C \Vert v\Vert_{L^\infty}^{2/3}$. In addition, they
proved that the order of vanishing can  be estimated
by this quantity 
for any solution of the elliptic equation
$\partial_{i}(a_{ij}\partial_{j}u)+ v u=0$
(cf.~also \cite{KSW} for stronger results in the plane).

In this paper, we address  this question for the parabolic
equation \eqref{EQ50}. 
As it is the case for the elliptic equation, it is an open
problem whether, under natural growth conditions at
infinity (for example as those in the present paper), 
the quantity \eqref{EQ51} provides
an upper bound for the order of vanishing for solutions $u$ 
of \eqref{EQ50}. Our main result asserts that the order of vanishing
bounded by
  \begin{equation}
   C \Vert w\Vert_{L^\infty}^{2}
   + C \Vert v\Vert_{L^{\infty}}^{2/3}
   \period
   \label{EQ56}
  \end{equation}
This result is in agreement with an elliptic upper bound proved
by Bourgain and Kenig in \cite{BK} when $w=0$. However, our proof is based
on a completely completely different approach, which we describe next.

The method presented here
is based on the parabolic frequency function introduced by
Kurata \cite{Kur} and Poon \cite{P}, 
which is in turn
inspired by
earlier works on elliptic equations
by Almgren \cite{Al} and Garofalo and Lin \cite{GL}.
It is based on the 
observation that
the frequency
  \begin{equation}
  \phi(t)
  =
    \int u^2(x,t) G_0(x,t)\,dx
   \label{EQ52}
  \end{equation}
is logarithmically convex for solutions of the heat equation.
Above,
  \begin{equation}
   G_{0}(x,t)
   =
   \frac{1}{t^{n/2}}
    e^{{|x|^2/4 t}}
   \comma x\in{\mathbb R}^{n}
   \commaone t>0
   \label{EQ10}
  \end{equation}
denotes the $(2\pi)^{n/2}$-multiple of a backward Gaussian kernel.
It was also shown in \cite{Kur,P}
that the method can be applied to the equation
\eqref{EQ50} yielding a strong unique continuation
property for $u$ when $v$ and $w$ are bounded.
Following the dependence on $\Vert v\Vert_{L^\infty}$ and
$\Vert w\Vert_{L^\infty}$, we obtain that
the degree of vanishing 
(for example under an assumption of periodicity) 
is bounded by
  \begin{equation}
   C \Vert v\Vert_{L^\infty}^2
    +C 
   e^{C \Vert w\Vert_{L^\infty}}
   \period   
   \label{EQ54}
  \end{equation}
In order to improve this result
and obtain the bound \eqref{EQ56},
we use a similarity variable
approach
from \cite{Ku2} (cf.~also \cite{An,Ch}) with the addition 
of a change
of variable
introduced in \cite{Ku4}
which optimizes the Dirichlet quotient
(i.e., replaces $Q$ in \eqref{EQ53}
with  $\bar Q$ in \eqref{EQ18}).
The similarity change of variable
was used in many contexts; here the idea is that
the parabolic structure and the Dirichlet quotient
method (\cite{A,AN,CFNT,FS}) leads to 
the necessary logarithmic convexity.
In addition, we use in an essential way theorems
due to 
Alessandrini and Vessella \cite{AV}
on the polynomial approximation
of a solution of a parabolic equation
(cf.~also \cite{H1,H2}).

The paper is structured as follows.
In Section~\ref{sec2}, we state the main result on the maximal
order of vanishing for the equation \eqref{EQ50}.
In Section~\ref{sec3} we find a point $x_\epsilon$
and the time $-\epsilon$, where 
in a certain sense the frequency is optimized.
In Section~\ref{sec4}, we recall the similarity change of variables
and   spectral properties of the resulting linear part.
Finally, Section~\ref{sec5} contains the proof of the main theorem.

\startnewsection{Notation and the main result}{sec2}
Our goal is to address the quantitative uniqueness for
solutions $u$ of the equation
  \begin{align}
   \partial_t u - \Delta u = w_j (x,t) \partial_j u + v(x,t) u \label{EQ11}
  \end{align}
defined for 
$(x,t)\in \mathbb{R}^n \times I$,
where $I$ is an open interval containing
$[T_0, T_0 +T]$,
with
$T_0 \in \mathbb{R}$ and $T>0$. 
The solution $u$
as well as the coefficients $v,w\in L^{\infty}({\mathbb R}^{n})$ 
are assumed to be
periodic in $x$ with respect to $\Omega = [-\pi, \pi]^n$
(for the conditions omitting the periodicity assumption,
cf.~Remark~\ref{R01} and Section~\ref{sec6} below).
Also, assume  that
  \begin{align}
   |w_j (x,t)| \leq M_1   
   \comma (x,t) \in \mathbb{R}^n \times I
   \commaone j = 1, \ldots, n
  \label{EQ12}
  \end{align}
 and 
  \begin{align}
   |v(x,t)| \leq M_0
   \comma (x,t) \in \mathbb{R}^n \times I
   \period
  \label{EQ13}
  \end{align}
Since we are interested in the dependence 
of the order of vanishing
on 
$M_0$ and $M_1$ 
when they are large, we assume $M_0, M_1 \geq 1$.
For any $(x_0, t_0) \in \mathbb{R}^n \times \mathbb{R}$ and $r>0$,
denote by 
  \begin{align*}
   Q_r (x_0 , t_0) 
     = \bigl\{
            (x,t) \in \mathbb{R}^n \times \mathbb{R}: |x- x_0|<r, -r^2 < t-t_0< 0 
       \bigr\},
  \end{align*}
the parabolic cylinder centered at $(x_0,t_0)$ with radius $r>0$,
while the parabolic norm of $(x,t) \in \mathbb{R}^n \times \mathbb{R} $  is given by
  \begin{align*}
   |(x,t)| = (|x|^2 + |t|)^{1/2} 
   \period
  \end{align*}
We write $W^{2,1}_{\infty} (Q_1)$
for the Sobolev space of functions whose $x$-derivatives up to second order and $t$-derivative up to the first order belong to $L^{\infty} (Q_1)$.
 
Denote by $O_{(x_0, t_0)} (u)$ the order of vanishing of $u$ at $(x_0,
t_0)$,
which we define (in the $L^2$ sense) as the
smallest integer $d$ such that
  \begin{equation}
   \Vert u\Vert_{L^2(Q_r(x_0,t_0))}
    =
   {\mathcal O}(r^{d+(n+2)/2}),
    \mbox{~~as~$r\to 0^+$}
   \period
   \label{EQ57}
  \end{equation}
Let
$q_0$ be an upper bound for the Dirichlet quotients of $u$ on
$[T_0,T_0+T]$, i.e.,
  \begin{equation}
   \frac{
    \Vert \nabla u(\cdot,t)\Vert_{L^2}^2
       }{
    \Vert u(\cdot,t)\Vert_{L^{2}}^2  
   }
   \le
   q_0
   \comma t\in [T_0,T_0+T]
   \period
   \label{EQ09}
  \end{equation}
The following is the main theorem of this paper.

\cole
\begin{Theorem}\label{thm1}
Let $u \in W^{2,1}_{\infty} (\Omega \times I)$ be a 
nontrivial solution of \eqref{EQ11} for $t$
in a neighborhood of $[T_0,T_0 +T]$
with $w_j$ and $v$ satisfying \eqref{EQ12} and~\eqref{EQ13},
respectively. Then the order of vanishing of $u$ at $(x_0, t_0)$
satisfies
  \begin{align}
   O_{(x_0,t_0)} (u) \leq C ( M_1 ^2 + M_0 ^{2/3} ) \label{vo11}
  \end{align}
for every $(x_0,t_0)\in \Omega\times[T_0+T/2,T_0+T]$
where the constant $C$ depends on $T$, $L$, and $q_0$.
\end{Theorem}
\colb

\begin{Remark}
\label{R01}
{\rm
It is not difficult to check that it is  possible to  replace
the $x$-periodicity assumption with
  \begin{equation}
   \int_{{\mathbb R}^{n}} u(x,t)^2\,dx
   \le
   M
   \int_{B_1} 
     u(x,t)^2
   \,dx
   \comma 
   t\in[T_0,T_0+T]
   \label{EQ55}
  \end{equation}
where $M$ is a constant,
making an assertion about the doubling at the point $(0,T_0+T)$.
For the necessary modifications, cf.~Section~\ref{sec6}
below.
}
\end{Remark}

The proof of Theorem~\ref{thm1} is divided into several lemmas. Let 
$(x_0,t_0) \in \Omega \times [T_0 + T/2,T_0 +T]$. 
By translation and rescaling, 
we may assume that $u$ is defined
for $t$ in an open interval $I$ of
$[-T,0]$ and that
  \begin{align*}
   (x_0, t_0) = (0,0)
   \period
  \end{align*}

\startnewsection{Optimizing the frequency function}{sec3}
Recall from \eqref{EQ10} that $G_0$ represents the
$(4\pi)^{n/2}$-multiple of the backward Gaussian kernel.

\cole
\begin{Lemma}\label{L01}
Let $u$ be as above
with
$v$ and $w_j$ satisfying  \eqref{EQ12} and~\eqref{EQ13}. 
Assume that $\epsilon\in(0,T)$.
Then
  \begin{align}
   \frac{\epsilon \int_{\mathbb{R}^n} |\nabla u(x_{\epsilon} +y, -\epsilon)|^2 G_0 (y, -\epsilon) \,dy }{\int_{\mathbb{R}^n}u(x_{\epsilon} +y, -\epsilon)^2 G_0 (y, -\epsilon) \,dy} \leq  \epsilon q(-\epsilon)
    \label{vo12}
  \end{align}
for some $x_{\epsilon} \in \Omega$.
\end{Lemma}
\colb

This lemma shall be used below with
$
   \epsilon = 1/ C(M_1 ^2 + M_0 ^{2/3})
$ 
where $C$ is sufficiently large.

\begin{proof}[Proof of Lemma~\ref{L01}]
Assume that, for some $\lambda >0$, we have 
  \begin{align}
  \lambda \int_{\mathbb{R}^n} u(x+y, -\epsilon)^2 G_0 (y, -\epsilon) \,dy 
     \leq 
    \int_{\mathbb{R}^n} |\nabla u (x+y, -\epsilon)|^2 G_0 (y, -\epsilon)\,dy 
   \comma x\in{\mathbb R}^{n}
   \period
   \label{EQ37}
  \end{align}
We intend to show that this is not possible if
$\lambda>q(-\epsilon)$
(this is sufficient since due to periodicity and continuity,
the minimum of the quotient in \eqref{vo12}
is achieved).
The statement \eqref{EQ37} is equivalent to
  \begin{align}
   \int_{\mathbb{R}^n} u(y, -\epsilon)^2 G_0 (y-x, -\epsilon) \,dy 
   \leq 
   \frac{1}{\lambda}
   \int_{\mathbb{R}^n} |\nabla u(y, -\epsilon)|^2 G_0 (y-x, -\epsilon) 
    \,dy
   \comma x\in{\mathbb R}^{n}
   \period
   \label{EQ38}
  \end{align} 
Integrating the left side of \eqref{EQ38} over $\Omega$ 
and using $\int_{\mathbb{R}^n} G_0 (y-x, -\epsilon)\,dx = (4 \pi)^{n/2}$, 
we obtain
  \begin{align}
   & \int_{\Omega} \int_{\mathbb{R}^n} u(y, -\epsilon)^2 G_0 (y-x, -\epsilon) \,dy \,dx 
    = \int_{\Omega} \sum_{j \in \mathbb{Z}^n} \int_{j+\Omega} u(y, -\epsilon)^2 G_0 (y-x, -\epsilon) \,dy \,dx 
   \nonumber\\&\indeq
     = \int_{\Omega} \sum_{j \in \mathbb{Z}^n} \int_{\Omega} u(y', -\epsilon)^2 G_0 (y' +j-x, -\epsilon)\,dy' \,dx
   \nonumber\\&\indeq
      = \int_{\Omega}
          u(y', -\epsilon)^2 \left( \sum_{j \in \mathbb{Z}^n} \int_{\Omega} G_0 (y' +j -x, -\epsilon)\,dx\right) \,dy' 
    \nonumber\\&\indeq
 = \int_{\Omega} u(y',-\epsilon)^2 \int_{\mathbb{R}^n} G_0 (y'-x , -\epsilon)\,dx \,dy' 
\nonumber\\&\indeq
    = (4 \pi)^{n/2} \int_{\Omega} u(y', -\epsilon)^2 \, dy'
   \period
   \label{EQ40}
  \end{align}
Similarly, we have
  \begin{equation}
      \int_{\Omega} \int_{\mathbb{R}^n} |\nabla u(y, -\epsilon)|^2 G_0 (y-x, -\epsilon) \,dy \,dx 
     =
     (4 \pi)^{n/2} \int_{\Omega} |\nabla u(y', -\epsilon)|^2 \,dy'
   \period
   \label{EQ39}
  \end{equation}
Combining \eqref{EQ40} and \eqref{EQ39} with \eqref{EQ38}, we get
  \begin{align}
   \lambda \int_{\Omega} u(y, -\epsilon)^2 \,dy \leq \int_{\Omega} |\nabla u(y, -\epsilon)|^2 \, dy = q(-\epsilon) \int_{\Omega} u(y, -\epsilon)^2 \,dy
   \period
  \end{align}
Since $u(\cdot, -\epsilon)$ is not identically zero,  we obtain a contradiction 
if $\lambda >q(-\epsilon)$. Setting
$\lambda = 2q(-\epsilon)$, we conclude that there exists $x_{\epsilon} \in \mathbb{R}^n$ such that \eqref{vo12}  holds. By periodicity of $u$, we may assume that $x_{\epsilon} \in \Omega$.
\end{proof}


We proceed with a change of variables
so that at time $-\epsilon$ the solutions starts at the point 
$x_{\epsilon}$ from Lemma~\ref{L01}.
Let 
  \begin{align}
   u(x,t) = \tilde{u} \left(x - \frac{x_{\epsilon}}{\epsilon}t,t\right)
   \period
  \end{align}
We then have 
  \begin{align}
     \tilde{u} (x,-\epsilon) = u(x+ x_{\epsilon}, -\epsilon)
  \end{align}
and
  \begin{align}
     \tilde{u} (x,0) = u(x,0)
   \period
  \end{align}
Note also that 
  \begin{align}
   O_{(0,0)}(\tilde{u}) = O_{(0,0)} (u) 
   \period
   \label{vo13}
  \end{align}
Furthermore,
  \begin{align}
   \frac{\int_{\Omega} |\nabla \tilde{u}(x,t)|^2 \,dx}{\int_{\Omega} \tilde{u}(x,t)^2 \,dx} = \frac{\int_{\Omega} |\nabla u(x,t)|^2 \,dx}{\int_{\Omega} u(x,t)^2 \,dx} = q(t),
  \end{align}
for all $t \in [-\epsilon, 0)$, and~\eqref{vo12} may be written as
  \begin{align}
  \frac{\epsilon \int_{\mathbb{R}^n} |\nabla \tilde{u} (y, -\epsilon)|^2 G_0 (y, -\epsilon)\,dy}{\int_{\mathbb{R}^n} \tilde{u}(y, -\epsilon)^2 G_0 (y, -\epsilon)\,dy} \leq  \epsilon q(-\epsilon)
   \period   
 \label{vo14}
  \end{align}
The function $\tilde{u}$ solves a modified equation
  \begin{align}
   \partial_t \tilde{u} - \Delta \tilde{u} 
        = -\frac{1}{\epsilon} x_{\epsilon} \cdot \nabla \tilde{u} 
         + w \cdot \nabla \tilde{u} + v \tilde{u}
   \period
   \label{EQ2.1.1}
  \end{align}
Observing \eqref{vo13}, we  remove tilde from here on, and write
$u$ instead of $\tilde u$.
Setting 
  \begin{align*}
   a = -\frac{x_{\epsilon}}{\epsilon},
  \end{align*}
\eqref{EQ2.1.1} may then be rewritten as 
  \begin{align}
   \partial_t u - \Delta u = a_j \partial_j u + w_j \partial_j u + v u
   \period
  \label{EQ2.1.1a}
  \end{align}

\startnewsection{Similarity variables}{sec4}

Next, we apply a similarity change of variable
  \begin{align}
   u(x,t) = e^{|x|^2 / 8(-t)} U \left(\frac{x}{\sqrt{-t}}, -\log(-t)\right)
  \end{align}
i.e.,
  \begin{align}
   U(y,\tau) = e^{-|y|^2 /8} u(y e^{- \tau/2}, - e^{-\tau})
  \end{align}
with $\tau = -\log(-t)$, i.e., $t= -e^{-\tau} $
for $t \in [-\epsilon, 0]$. Also, write
  \begin{align}
   V(y,\tau)= v(y e^{-\tau /2}, -e^{-\tau})
   \comma 
      (y, \tau) \in \mathbb{R}^n \times [\tau_0 , \infty) 
  \end{align}
and
  \begin{align}
   W_j (y, \tau) = w_j (y e^{- \tau /2} , -e^{- \tau}) 
   \comma (y,\tau) \in \mathbb{R}^n \times [\tau_0 , \infty),
  \end{align}
for $j=1, \ldots,n$ where $\tau_0 =  \log(1/ \epsilon)$. 
Denoting
  \begin{align}
   HU = - \Delta U + \left( \frac{|y|^2}{16} - \frac{n}{4}\right)U
  \end{align}
and
  \begin{align}
   Q (\tau) = \frac{(HU, U)_{L^2 ({\mathbb{R}^n})}}{\|U\|^2_{L^2 ({\mathbb{R}^n})}},
   \label{EQ53}
  \end{align}
the equation \eqref{EQ2.1.1a} may be written as
  \begin{align}
   \partial_{\tau} U + HU & = e^{-\tau /2} (a_j y_j U + a_j \partial_j U) + e^{- \tau /2} (y_j W_j U + W_j \partial_{\tau} U) 
   + e^{-\tau} V U \label{EQ14}  
  \end{align}
for $\tau \geq \tau_0$, while at $\tau = \tau_0$,
  \begin{align}
   U(y,\tau_0) & = e^{-|y|^2 / 8} u \left(\frac{y}{\sqrt{\epsilon}}, -\epsilon\right)
   \period
 \label{EQ15}
  \end{align}
Then
  \begin{align}
   & \partial_{\tau} U + \left( A(\tau) - \bar{Q} (\tau) I \right) U + \bar{Q} (\tau) U \nonumber \\ & \indeq =  e^{-\tau /2} a_j \partial_j U + e^{- \tau /2}( y_j W_j U + W_j \partial_j U ) + e^{-\tau} V U \label{EQ16}
  \end{align}
where we denote
  \begin{align}
   A(\tau)U = HU - e^{-\tau /2} a_j y_j U \label{EQ17}
  \end{align}
and 
  \begin{align}
   \bar{Q}(\tau) & = \frac{(A(\tau)U,U)_{L^2(\mathbb{R}^n)}}{\|U\|^2_{L^2 ({\mathbb{R}^n})}} 
   = Q (\tau) - \frac{e^{- \tau /2} a_j}{\|U\|^2_{L^2 (\mathbb{R}^n)}} \int_{\mathbb{R}^n} y_j U^2 \,dy
   \period
  \label{EQ18}
  \end{align}
A straight-forward change of variables yields 
  \begin{align}
   \norm{U(\cdot,\tau)}^2_{\Lnorm} = \int_{\mathbb{R}^n} u(x,t)^2 G_0 (x,t) \,dx \label{EQ18'}
  \end{align}
and 
  \begin{align}
   (HU(\cdot,\tau), U(\cdot,\tau))_{\Lnorm} = |t| \int_{\mathbb{R}^n} |\nabla u(x,t)|^2 G_0 (x,t) \,dx 
   \period
\label{EQ18''}
  \end{align}
We now observe a simple fact
  \begin{align}
   \norm{(H - \bar{Q}(\tau)I)v}_{\Lnorm} \geq \dist (\bar{Q}(\tau), \sp(H))
   \comma 
   \Vert v\Vert_{L^2({\mathbb R}^{n})}=1
   \period
   \label{EQ41}
  \end{align}

\noindent In order to analyze the asymptotic behavior of eigenvalues of $A(\tau)$, 
we first recall the spectral properties of $H$. 
The Hermite polynomials $h_k$ are defined on the real line as
  \begin{align*}
   h_k (x) = (-1)^k e^{x^2} \frac{d^k}{dx^k} e^{-x^2},
  \end{align*}
while the Hermite functions read
  \begin{align*}
   \tilde{h}_k (x) = h_k (x) e^{-x^2 /2}
   \period
  \end{align*} 
By taking the product of one-dimensional Hermite functions, we
generalize the definition to $\mathbb{R}^n$,
  \begin{align}
   \phi_{\alpha} (x) = \prod_{k=1} ^{n} \tilde{h}_{\alpha_k} (x_k),
  \end{align}
where 
$\alpha= (\alpha_1,\ldots, \alpha_n)\in \mathbb{N}_0^n$, and $x= (x_1, \ldots, x_n)\in \mathbb{R}^n$. The set 
  \begin{align}
   \{\phi_{\alpha} : \alpha \in \mathbb{N}_0^n \}
  \end{align}
forms a complete orthonormal system for $\Lnorm$ and the functions $\phi_{\alpha}$ solve
  \begin{align}
   (\Delta - |x|^2) \phi_{\alpha} (x) = -(2|\alpha| +n) \phi_{\alpha}
   \period
  \end{align}
Then we have
  \begin{align}
   H \left(\phi_{\alpha} \left(\frac{x}{2} \right) \right)
      = \frac{|\alpha|}{2} \phi_{\alpha} \left(\frac{x}{2}\right)
   \comma \alpha\in{\mathbb N}_0^{n}
  \end{align}
and thus
  \begin{align}
   \sp(H) = \left\{ \frac{m}{2}: m \in \mathbb{N}_0 \right\}
   \period
  \label{EQ18.0}
  \end{align}

\cole
\begin{Lemma} \label{L03}
For a sufficiently large constant $K_0 >0$, set
  \begin{align}
   \epsilon = \frac{1}{K_0 (M_1 ^2 + M_0 ^{2/3})}
   \period
  \end{align} 
Then
  \begin{align}
   \bar{Q} (\tau) = {\mathcal O}\left(\frac{1}{\epsilon}\right),
    \hbox{~~as~} \tau \rightarrow \infty
   \period
  \end{align}
Also,
  \begin{align}
   \bar{Q} (\tau) \to \frac{m}{2},
    \hbox{~~as~} \tau \rightarrow \infty
  \end{align}
for some $m \in \mathbb{N}_0$ for which
  \begin{equation}
   m
   \le
   C (M_0^{2/3}+M_1^2)
   \label{EQ35}
  \end{equation}
where $C$ depends on $q_0$. Moreover,
for every $\delta>0$, there exist
$\eta\in(\log(1/\epsilon),0)$ and constants
$A_1(\delta),A_2(\delta)>0$ such that
  \begin{align}
   A_1(\delta) |t|^{m+\delta} \leq \int u^2 (x,t) G_0 (x,t)\,dx \leq A_2(\delta) |t|^{m-\delta}
   \label{EQ42}
  \end{align}
for all $t\in [-\eta,0)$.
\end{Lemma}
\colb

\begin{proof}[Proof of Lemma~\ref{L03}]
We first divide \eqref{EQ16} by $\Unormlong$ and then take the inner product 
of the resulting equation
with $(A(\tau) - \bar{Q} (\tau)I) U/{\Unormlong}$. Let $B(U)$ denote the right side of \eqref{EQ16}. 
We obtain
  \begin{align}
   &  \left( \frac{\partial_{\tau} U}{\Unorm} , (A(\tau)- \bar{Q}(\tau)I)\frac{U}{\Unorm} \right) + \norm[\bigg]{ (A(\tau)- \bar{Q}(\tau)I)\frac{U}{\Unorm} }^2 + \bar{Q} (\tau) \frac{(U, A(\tau)U)}{\Unorm^2} \nonumber \\ 
   & \indeq = {\bar{Q}}^2 (\tau)  
    +
    \left( \frac{B(U)}{\Unorm}, (A(\tau)- \bar{Q}(\tau)I)\frac{U}{\Unorm} \right)
   \label{EQ19}
  \end{align}
where we abbreviate
  \begin{equation*}
   (\cdot,\cdot)=(\cdot,\cdot)_{L^2({\mathbb R}^{n})}
  \end{equation*}
and
  \begin{equation*}
   \Vert\cdot\Vert
    =\Vert \cdot\Vert_{L^2({\mathbb R}^{n})}
   \period
  \end{equation*}
Note that
  \begin{align}
   \bar{Q} (\tau) \frac{(U, A(\tau)U)}{\Unorm^2} = {\bar{Q}}^2 (\tau),
  \end{align}
so \eqref{EQ19} becomes 
  \begin{align}
   & \left( \frac{\partial_{\tau} U}{\Unorm} , (A(\tau)- \bar{Q}(\tau)I)\frac{U}{\Unorm} \right) + \norm[\bigg]{(A(\tau)- \bar{Q}(\tau)I)\frac{U}{\Unorm}}^2  \nonumber \\ 
    & \indeq = \left( \frac{B(U)}{\Unorm}, (A(\tau)- \bar{Q}(\tau)I)\frac{U}{\Unorm} \right)
   \period
   \label{EQ20}
  \end{align}
On another note, by differentiating \eqref{EQ18}, we get
  \begin{align}
   {\bar{Q}}^{\prime} (\tau) = 
       \left(  2 \frac{(\partial_{\tau} U, A(\tau)U)}{\Unorm^2}
         + \frac{(U, A'(\tau)U)}{\Unorm^2} - \frac{2(\partial_{\tau}U,U)(A(\tau)U,U)}{\Unorm^{4}}\right)
   \period
   \label{EQ32}
  \end{align}
Since $A(\tau)$ is a symmetric operator whose derivative is given by
  \begin{align}
   A'(\tau)U & = \frac{1}{2} e^{-\tau /2} a_j y_j U 
  \end{align}
we obtain
  \begin{align}
   \frac{1}{2} \bar{Q}^{\prime} (\tau) & =  \frac{(\partial_{\tau} U, A(\tau)U)}{\Unorm^2} + \frac{1}{4} e^{-\tau /2} \frac{(U, a_j y_j U)}{\Unorm^2} - \frac{(\partial_{\tau}U, U)(A(\tau)U, U)}{\Unorm^{4}} \nonumber \\
  & = \frac{(\partial_{\tau}U, (A(\tau)- \bar{Q} (\tau)I)U)}{\Unorm^2} + \frac{1}{4} e^{-\tau /2} \frac{(U,a_j y_j U)}{\Unorm^2}
   \period
 \label{EQ21}
  \end{align}
Then, substituting \eqref{EQ21} in \eqref{EQ20}, we get
  \begin{align}
    & \frac{1}{2} \bar{Q}^{\prime} (\tau) + \norm[\bigg]{(A(\tau)- \bar{Q}(\tau)I) \frac{U}{\Unorm}}^2  \nonumber \\
    & \indeq = \frac{1}{4} e^{-\tau /2} \frac{(a_j y_j U,U)}{\Unorm^2} + \frac{e^{-\tau /2} a_j}{\Unorm^2} (\partial_j U, (A(\tau)- \bar{Q}(\tau)I)U) \nonumber \\
    & \indeq\indeq + \frac{1}{\Unorm^2} \left( e^{-\tau /2} (y_j W_j U + W_j \partial_j U)+ e^{-\tau} V U, (A(\tau)- \bar{Q} (\tau)I)U \right)
   \period
   \label{EQ22}
   \end{align}
We integrate by parts on the second term on the right hand side and note that 
$\int U \partial_j U =0$ and $\int \Delta U \partial_j U =0$.
Therefore, \eqref{EQ22} becomes
  \begin{align}
    & \frac{1}{2} \bar{Q}^{\prime} (\tau) + \norm[\bigg]{(A(\tau)- \bar{Q}(\tau)I) \frac{U}{\Unorm}}^2 \nonumber \\
    & \indeq = \frac{e^{-\tau /2} a_j}{4\Unorm^2} \int y_j U^2 \,dy - \frac{e^{-\tau /2} a_j}{16 \|U\|^2} \int y_j U^2 \,dy \nonumber 
    + \frac{e^{- \tau} |a|^2}{2} 
    \\& \indeq\indeq 
        + \frac{e^{-\tau /2}}{\Unorm^2} \int 
              \Bigl(y_j W_j U + W_j \partial_j U + e^{-\tau/2} V U
              \Bigr)(A(\tau) -\bar{Q} (\tau) I)U \,dy. \label{EQ23}
    \nonumber\\&\indeq
     = I_1+I_2+I_3+I_4
   \period
  \end{align}
Now, 
  \begin{align}
   I_1 &= \frac{e^{-\tau /2}}{4 \Unorm^2} a_j\int y_j U^2 \,dy 
     \leq \frac{e^{-\tau /2}}{4 \Unorm^2} |a| \int  |y| U^2 \,dy \\
    & \leq \frac{e^{-\tau /2}}{4 \Unorm^2} |a| {\left(\int  |y|^2 U^2 \,dy \right)}^{1/2},
  \end{align}
and similarly
  \begin{align}
   I_2 +I_3
     \leq \frac{e^{-\tau /2}}{16 \Unorm} |a| {\left( \int  |y|^2 U^2 \,dy \right)}^{1/2} 
     + \frac{ e^{-\tau} |a|^2}{2}
   \period
  \end{align}
As for the last term in \eqref{EQ23}, 
  \begin{align}
   I_4 
    \leq \norm[\bigg]{(A(\tau)- \bar{Q} (\tau)I)\frac{U}{\Unorm}} \left( \frac{e^{-\tau /2}}{\norm{U}} \norm{y_j W_j U + W_j \partial_j U }+ e^{-\tau}\norm{V U}\right)
   \period
  \end{align}
Hence, for the right side of \eqref{EQ23} we have
  \begin{align}
   &I_1 + I_2 + I_3 +I_4
    \nonumber\\&\indeq
    \leq C \frac{e^{-\tau /2} |a|}{\Unorm}  {\left( \int |y|^2 U^2 \,dy \right)}^{1/2} + Ce^{-\tau} |a|^2  
    \nonumber\\&\indeq\indeq
      + e^{-\tau /2} \norm[\bigg]{(A(\tau) 
         - \bar{Q} (\tau)I)\frac{U}{\Unorm}}
      \left( \frac{M_1}{\Unorm} 
       (\Vert y U\Vert
        + \Vert \nabla U\Vert)
        + e^{-\tau /2} M_0 
      \right)
   \period
  \end{align}
It is easy to check that
  \begin{align}
    \Vert \nabla U\Vert^2
    + \Vert y U\Vert^2
    \leq C \left( \bar{Q} (\tau)_{+} + e^{-\tau} |a|^2 +1 \right) 
    \Vert U\Vert^2
    \label{EQ24}
  \end{align}
where
  \begin{equation}
   x_{+}=\max\{x,0\}
   \comma x\in{\mathbb R}
   \period
   \label{EQ67}
  \end{equation}
Consequently, 
  \begin{align}
   & \frac{1}{2} \bar{Q}^{\prime} (\tau) + \norm[\bigg]{(A(\tau)- \bar{Q}(\tau)I) \frac{U}{\Unorm}}^2 \nonumber \\
   & \indeq \leq C e^{-\tau /2} |a| 
         \left( \bar{Q} (\tau)_++ e^{-\tau} |a|^2 +1 \right)^{1/2} + C e^{-\tau} |a|^2 \nonumber \\
   & \indeq\indeq + C e^{-\tau /2} \norm[\bigg]{(A(\tau)- \bar{Q}(\tau)I) \frac{U}{\Unorm}} \left(M_1 (\bar{Q} (\tau)_+ + e^{-\tau} |a|^2 +1)^{1/2}  + e^{-\tau /2} M_0 \right) \label{EQ25}
   \period
  \end{align} 
Applying Young's inequality to the last term, we get 
  \begin{align}
    \bar{Q}^{\prime} (\tau) 
       & \leq C e^{-\tau /2} |a| \left(\bar{Q} (\tau)_+ + e^{-\tau} |a|^2 +1 \right)^{1/2} + C e^{-\tau} |a|^2
   \nonumber\\&
    \indeq + C e^{-\tau} \left( M_1 ^2 (\bar{Q} (\tau)_+  + e^{-\tau} |a|^2 +1) + e^{-\tau} M_0 ^2 \right)
   \nonumber\\&
    \leq C e^{-\tau} M_1 ^2 \bar{Q} (\tau)_+ + C e^{-\tau /2} |a| (\bar{Q} (\tau)_+)^{1/2} + C e^{-\tau /2} |a|
   \nonumber\\&\indeq
   + C e^{-\tau} (|a|^2 + M_1 ^2) + C e^{-2\tau} \left( M_1 ^2 |a|^2 + M_0 ^2 \right) 
   \period
  \end{align}
Observe that
  \begin{align}
   C e^{-\tau_0} M_1 ^2   = C e^{- \log (1/\epsilon)} M_1 ^2
    \leq \frac{C M_1 ^2}{K_0 (M_1 ^2 + M_0 ^{2/3})} \leq \frac{1}{2},
  \end{align}
where the last inequality is obtained by assuming that
$K_0$ is a sufficiently large constant. 
Thus we may now apply \cite[Lemma~2.2]{Ku1}
and estimate 
  \begin{align}
   \bar{Q} (\tau) \leq C \left( \max (\bar{Q} (\tau_0), 0) + M_1 ^2 |a|^2 {\epsilon}^2 + M_1 ^2 {\epsilon} + M_0 ^2 {\epsilon}^2 + |a|^2 \epsilon + |a|{\epsilon}^{1/2} \right)
   \comma \tau\ge \tau_0
   \period
  \end{align}
By \eqref{vo14}, \eqref{EQ18}, \eqref{EQ18'}, and \eqref{EQ18''}
we have
  \begin{align}
   \bar Q(\tau_0)
   &\le
   Q(\tau)
   +
   \left|
    \frac{
     e^{-\tau_0/2}a_j
        }{
     \Vert U\Vert^2
    }
   \int y_j U^2
       \,dy
   \right|
    \nonumber\\&
    \le
    2\epsilon q(\epsilon)
    +
    \frac{C\sqrt{\epsilon}}{\epsilon}
    \frac{
     (H U(\tau_0),U(\tau_0))^{1/2}
        }{
     (U(\tau_0),U(\tau_0))^{1/2}
    }
   \le
    2\epsilon q(\epsilon)
    +
   \frac{C}{\sqrt{\epsilon}}
     (\epsilon q(-\epsilon))^{1/2}
   \nonumber\\&
   \le
    2\epsilon q(\epsilon)
    +
   C q(-\epsilon)^{1/2}
   \le
   2\epsilon
   q_0
   +
   C q_0^{1/2}
   \label{EQ36}
  \end{align}
where we used \eqref{EQ09} in the last inequality.
Recall that $a = -{x_{\epsilon}}/{\epsilon}$, where $x_{\epsilon} \in \Omega$,
which implies $|a| \leq C/ \epsilon$. Therefore,
  \begin{align}
   \bar{Q}(\tau) 
    &\leq C \left( \epsilon q_0+q_0^{1/2}+ M_1 ^2 + 1+ M_0 ^{2/3} + \sqrt{M_1 ^2 + M_0 ^{2/3}} \right)
    \nonumber\\&
    \le
    C \left( q_0+ M_1 ^2 + 1+ M_0 ^{2/3}  \right)  
    \le
    C \left( M_1 ^2 + 1+ M_0 ^{2/3}  \right)  
    \nonumber\\&
    = K_{M_0, M_1}
   \label{EQ26}
  \end{align}
by allowing  $C$ to depend on $q_0$.
Going back to \eqref{EQ25} and applying Young's inequality one more time, we obtain 
  \begin{align}
   & \frac{1}{2} \bar{Q}^{\prime} (\tau) + \norm[\bigg]{(A(\tau)- \bar{Q}(\tau)I) \frac{U}{\Unorm}}^2  \nonumber \\ 
   & \indeq\leq C 
   e^{-\tau /2} |a| \left( \bar{Q} (\tau)_+ + e^{-\tau} |a|^2 + 1 \right)^{1/2} 
     + C e^{-\tau} |a|^2 
   \nonumber\\ &\indeq\indeq
   + Ce^{-\tau} 
        M_1 ^2 (\bar{Q} (\tau)_+ + e^{-\tau} |a|^2 +1)
   +  M_0 ^2 e^{-2\tau}
   \period
   \label{EQ27}
  \end{align}
Integrating \eqref{EQ27} and using the bound on $\bar{Q}$, we get 
  \begin{align}
   &\frac{1}{2} \left( \bar{Q} (\tau) - \bar{Q} (\tau_0)\right)  
         + \int_{\tau_0} ^{\tau} \norm[\bigg]{(A(s)- \bar{Q}(s)I)\frac{U(\cdot,s)}{\norm{U(\cdot,s)}}}^2 \,ds
  \nonumber\\&\indeq
  \leq C |a| \left(K_{M_0, M_1} ^{1/2} (\epsilon^{1/2} - e^{-\tau/2}) \right) + C|a|^2 (\epsilon - e^{-\tau}) \nonumber \\
   & \indeq\indeq + C \left( M_1 ^{2} K_{M_0, M_1} (\epsilon - e^{-\tau}) +M_1^2 |a|^2 (\epsilon^2 - e^{-2\tau}) \right) + C M_0 ^2 (\epsilon^2 - e^{-2\tau})
   \period
 \label{EQ28}
  \end{align}
In order to estimate the second term on the left from below, let
$\Vert v\Vert_{L^2}=1$.
Then
  \begin{equation}
   \left\Vert
     (    A(s)-\bar Q I) v
   \right\Vert^2
   \ge
   \frac12
   \left\Vert
     (   H-\bar Q I) v
   \right\Vert^2
   - 
   C e^{-\tau}
   |a|^2 \bar Q 
   -
   C e^{-\tau}
   |a|^2
   \period
   \label{EQ68}
  \end{equation}
Using this inequality in \eqref{EQ27}
and letting $\tau \to \infty$, we get
  \begin{align}
   \int_{\tau_0}^{\infty} 
      \norm[\bigg]{( H - \bar{Q}(s)I)\frac{U(\cdot,s)}{\norm{U(\cdot,s)}}}^2 \,ds<\infty
   \period
 \label{EQ29}
  \end{align}
Combining this fact with \eqref{EQ41}
and
  \begin{equation}
   \limsup_{\tau\to\infty}
   \bar Q'(\tau)
   \le 0,
   \label{EQ34}
  \end{equation}
which follows from \eqref{EQ27},
we get
  \begin{align}
   \mbox{dist} (\bar{Q}(\tau), \sp(A(\tau))) \to 0
  \end{align} 
as $\tau \to \infty$. 
Using 
\eqref{EQ18.0},
we finally obtain
  \begin{align}
   \bar{Q}(\tau) \to \frac{m}{2}
   \label{EQ33}
  \end{align}
as $\tau \to \infty$ for some 
$m \in \mathbb{N}_0$.
Going back to \eqref{EQ14}, we have
  \begin{align*}
    (\partial_{\tau} U, U) + \bar{Q} (\tau) 
        &= e^{-\tau /2} a_j (\partial_j U, U) + e^{-\tau /2} (y_j W_j U, U) 
     + e^{-\tau/2} (W_j \partial_j U, U) + e^{-\tau} (V U, U)
  \end{align*}
which may be rewritten as 
  \begin{align}
   \frac{1}{2} \frac{1}{\Unorm^2} \frac{d}{d{\tau}} \norm{U}^2 
       + \bar{Q} (\tau) 
   = 
   \frac{f(\tau)}{\norm{U}^2} \label{EQ30}
  \end{align}
where
  \begin{align*}
   \frac{f(\tau)}{\norm{U}^2} = e^{-\tau/2} a_j \frac{(\partial_j U,U)}{\norm{U}^2} + e^{-\tau/2} \frac{(y_j W_j U,U)}{\norm{U}^2} + e^{-\tau/2} \frac{(W_j \partial_j U, U)}{\norm{U}^2} + e^{-\tau} \frac{(V U,U)}{\norm{U}^2}
   \period
  \end{align*}
Note that $(\partial_j U,U)=0$, and thus
  \begin{align*}
   \frac{f(\tau)}{\norm{U}^2} & \le e^{-\tau/2} M_1 \left( \frac{\int |y|U^2\,dy + \int |\nabla U| |U|\,ds}{\norm{U}^2} \right) + e^{-\tau} M_0 \\
   & \leq 2e^{ -\tau/2} M_1 \left( \frac{\int |y|U^2\,dy + \int |\nabla U|^2\,ds +  \int |U|^2\,ds}{\norm{U}^2} \right) + e^{-\tau} M_0 \\
   & \leq C e^{-\tau/2}M_1 \left(|\bar{Q} (\tau)| + e^{-\tau} |a|^2 +1 \right) + e^{-\tau} M_0
   \period
  \end{align*}
As $\bar{Q}(\tau)$ is uniformly bounded,
  \begin{align*}
   \int_{\tau_1}^{\tau} \frac{f(s)\,ds}{\norm{U(\cdot,s)}^2} 
   \leq 
    C K_{M_0, M_1} M_1 (e^{-\tau_1 /2} - e^{-\tau/2}) + C M_1 |a|^2 (e^{-3\tau_1 /2} - e^{-3\tau/2}) + M_0 (e^{-\tau_1}- e^{-\tau})
   \period
  \end{align*}
Integrating \eqref{EQ30}, we get
  \begin{align*}
   & \frac{1}{2} \log{\norm{U(\cdot,\tau)}^2} - \frac{1}{2} \log{\norm{U(\cdot, \tau_1)}^2}  \\
   & \indeq = - \int_{\tau_1} ^{\tau} \bar{Q}(s)\,ds + \int_{\tau_1} ^{\tau} \frac{f(s)}{\norm{U(\cdot,s)}^2} \,ds \\
   & \indeq = -\frac{m}{2} (\tau - \tau_1) - \int_{\tau_1} ^{\tau} \left( \bar{Q}(s) -\frac{m}{2} \right)\,ds + {\mathcal O}(e^{-\tau_1 /2} - e^{-\tau /2}),
    \hbox{~~as~} \tau \to \infty
   \period
  \end{align*}
By \eqref{EQ28}, we have that for all $\delta >0$
there exists $t_1(\delta)>0$ such that
  \begin{align*}
   - \delta (\tau - \tau_1) \leq \log{\norm{U(\tau)}^2} - \log{\norm{U(\tau_1)}^2} + m(\tau - \tau_1) \leq \delta (\tau - \tau_1)
   \comma \tau \geq \tau_1 (\delta)
   \period
  \end{align*}
Therefore, 
  \begin{align*}
   e^{-\delta (\tau - \tau_1)} \leq \frac{e^{\tau m} \norm{U(\tau)}^2}{e^{\tau_1 m} \norm{U(\tau_1)}^2} \leq e^{\delta (\tau -\tau_1)}
   \comma \tau \geq \tau_1
   \period
  \end{align*}
Consequently, for all $\delta >0$, there exists $\tau_1 \geq \tau_0$ such that 
  \begin{align}
   A_1(\delta) e^{-\delta \tau} \leq e^{m \tau} \norm{U(\tau)}^2 \leq A_2(\delta) e^{\delta \tau} \label{EQ31}
   \comma \tau \geq \tau_1
  \end{align}
for some positive constants $A_1(\delta)$ and $A_2(\delta)$. 
Switching back to the original variables, we obtain 
\eqref{EQ42}, and the proof is concluded.
\end{proof}

A function $f(x,t)$ is homogeneous of degree $d$ if for any
$\lambda>0$ and $(x,t) \in \mathbb{R}^n \times \mathbb{R} \setminus \{(0,0)\}$
if it satisfies
  \begin{equation*}
   f(\lambda x, \lambda ^2 t) = \lambda ^d f(x,t)
   \period
  \end{equation*}
A polynomial $P(x,t)$ of degree at most $d \in \mathbb{N}$ 
can be decomposed into a sum of homogeneous polynomials 
whose degree of homogeneity is at most $d$.
The following elementary results regarding polynomials play an important role in the subsequent argument.

\cole
\begin{Lemma}\label{L04}
Let $Q = \sum_{|\mu|+ 2l = d} C_{\mu,l} x^{\mu} t^l$ 
be a homogeneous polynomial of degree $d\in{\mathbb N}_0$.
Then
  \begin{equation*}
   \int_{\mathbb{R}^n} Q(x,t) G_0 (x,t) \,dx
  = C_0 |t|^{d/2}
  \end{equation*}
where $C_0$ is a constant depending on the polynomial.
\end{Lemma}
\colb

\begin{proof}[Proof of Lemma~\ref{L04}]
By substituting $Q$ in the integral, we obtain
  \begin{align}
	&\int_{\mathbb{R}^n} Q (x,t) G_0 (x,t)\,dx  =  \sum_{|\mu|+ 2l = d} C_{\mu, l} \int_{\mathbb{R}^n} x^{\mu} t^l \frac{e^{-|x|^{2} / 4|t| }	}{|t|^{n/2}}\,dx
      \nonumber\\&\indeq
	 =  \sum_{|\mu|+ 2l =d} C_{\mu,l} |t|^{l+ |\mu|/2} \int_{\mathbb{R}^n} \frac{x^{\mu}}{|t|^{|\mu|/2}} \frac{e^{-|x|^{2} / 4|t| }	}{|t|^{n/2}} \,dx
       \nonumber\\&\indeq
	 =  \sum_{|\mu|+ 2l =d} C_{\mu,l} |t|^{l+ |\mu|/2} \int_{\mathbb{R}} \cdots \int_{\mathbb{R}} \left( \frac{x_1}{|t|^{1/2}}\right)^{\mu_1} \cdots\left( \frac{x_n}{|t|^{1/2}} \right)^{\mu_n}\frac{e^{-|x|^{2} / 4|t| }}{|t|^{n/2}} \,dx_1 \cdots dx_n
    \nonumber\\&\indeq
	 =  
          |t|^{d/2} 
          \sum_{|\mu|+ 2l =d} {C}_{\mu,l}
             \Gamma(\mu_1+1) \cdots \Gamma(\mu_n+1)   
  \end{align}	
and the lemma is established.
\end{proof}

\nnewpage
\cole
\begin{Lemma} 
\label{L05}
Consider the polynomial $ x^{\mu}t^l$ with $|\mu|+ 2l =
2d$ for some integer $d \geq 0$. 
If all the coordinates of 
$\mu = (\mu_1, \mu_2,\ldots, \mu_n)$ are even, then
for all $r>0$
  \begin{equation}
   \int_{B(0,r)} 
     x^{\mu} t^l G_0 (x,t) \,dx 
     = C_0 |t|^d + {\mathcal O}\left( |t|^{d+1} \right),
   \hbox{~~as~$t\to0-$}
   \label{EQ47}
  \end{equation} 
where $C_0$ is a constant.
If $\mu_i$ is an odd integer for some $i \in \{1,\cdots,n \}$, then 
  \begin{equation}
   \int_{B(0,r)} x^{\mu} t^l G_0 (x,t)\,dx = 0
   \period
   \label{EQ45}
  \end{equation}
\end{Lemma}
\colb

\begin{proof}[Proof of Lemma~\ref{L05}]
For $|\mu|>0$, we have 
  \begin{align}
    &
    \left|
    \int_{B(0,r)} 
      x^{\mu} t^l G_0 (x,t) \,dx 
    \right|
    \nonumber\\&\indeq
    \le  
    C
    |t|^{l+|\mu|/2}  
     \prod_{i=1}^{n} \int_{0}^{r} 
        \left(
          \frac{x_i}{|t|^{1/2}} 
        \right)^{\mu_i} 
       \frac{e^{-|x_i|^2 /4|t|}}{|t|^{1/2}} \,dx_i
    \nonumber\\&\indeq
     =
     K
     |t|^{d}
    \prod_{i=1}^n \int_0^{r^2/4|t|} z^{{\mu_i}/2 - 1/2} e^{-z} \,dz
   \label{EQ48}
  \end{align}
where 
$K$ is a fixed constant.
It is easy to check that
  \begin{align}
   \left|
   C_0
   -
    K
    \prod_{i=1}^n \int_0 ^{r^2/4|t|} z^{{\mu_i}/2 - 1/2} e^{-z} \,dz
   \right|
   = {{\mathcal O}} (|t|),
   \hbox{~~as~$t\to0-$}
   \label{EQ44}
  \end{align}
where
  \begin{equation}
   C_0
   =
    K
    \prod_{i=1}^n \int_0^{\infty} z^{{\mu_i}/2 - 1/2} e^{-z} \,dz
   \label{EQ43}
  \end{equation}
and \eqref{EQ47} follows.
If $\mu_i$ is odd for some $i$, then
we have \eqref{EQ45} by symmetry.
%
%
\end{proof}

\startnewsection{The proof of the main theorem}{sec5}

Before concluding with the proof of the main theorem, we 
need a statement
connecting the order of vanishing with the degree of the
eigenfunction.

\nnewpage
\cole
\begin{Lemma} \label{L06}
Consider a solution $u \in W_{\infty} ^{2,1} (\Omega \times I)$ to the equation
    \begin{align}
    \partial_t u + \Delta u = a_j \partial_j u + w_j(x,t) \partial_j u + v(x,t) u
    \label{EQ01}
  \end{align}
for $t \in \left[- \epsilon, 0 \right]$ with $\epsilon \in (0,1)$ where $a_j$ are constants and where $w_j$ and $v$ satisfy \eqref{EQ12} and~\eqref{EQ13},
and suppose
that $u$ has finite order of vanishing $d\in{\mathbb N}_0$ at
$(0,0)$. 
With $m\in{\mathbb N}$,
assume that 
for all $\delta >0 $
there exist $\eta\in(0,\epsilon]$ and 
$A_1(\delta), A_2(\delta)>0$ such that
    \begin{align}
    A_1(\delta) |t|^{m +\delta} \leq \int_{\mathbb{R}^n} u^2 (x,t) G_0 (x,t) \,dx 
      \leq A_2(\delta) |t|^{m- \delta}
   \comma t \in \left[ - \eta, 0 \right]
   \period
  \label{EQ02}
  \end{align}
Then $d= m$.
\end{Lemma}
\colb

\begin{proof}[Proof of Lemma~\ref{L06}]
Since the statement clearly holds when $d=0$, we assume $d\in{\mathbb N}$.
First, assume that \eqref{EQ02} holds
for some fixed $\delta>0$.
Fix $\delta>0$. Since the degree of vanishing of $u$ at $(0,0)$ is $d>0$,
there exists $\gamma>0$ such that
    \begin{align}
     |u(x,t)| 
     \leq 
     \gamma |(x,t)|^d  
     =
     \gamma\left(|x|^2 + |t|\right)^{d/2} \label{EQ03}
    \end{align}
for all $(x,t) \in Q_1$. Note that for any $R>0$ we have
  \begin{align}
   & \int_{\mathbb{R}^n \setminus B(0,R)} u^2 (x,t)G_0 (x,t) \,dx 
   \leq
    \Vert u\Vert_{L^{\infty} (\mathbb{R}^n)}^2 
      \int_R^{\infty} \int_{\partial B(0,r)} 
          \frac{e^{-r^2 /4|t|}}{|t|^{n/2}} \,dS_y 
      \,dr 
   \le
   C
   \Vert u\Vert_{L^\infty}^2
   \frac{|t|^{1/2}}{R}
   e^{-R^2/8|t|}
   \period
   \label{EQ46}
  \end{align}
By \eqref{EQ02} and \eqref{EQ46}, 
for all $\delta>0$ there 
exist $\eta\in(0,\epsilon)$
and $A_1(\delta),A_2(\delta)>0$ such that
  \begin{equation}
   \frac12 A_1(\delta) |t|^{m +\delta} \leq \int_{B(0,R)} u^2 (x,t) G_0 (x,t) \,dx \leq A_2(\delta) |t|^{m- \delta}
   \comma t\in[-\eta,0]
   \period
 \label{EQ04}
  \end{equation}
Furthermore, 
by Lemmas~\ref{L04} and~\ref{L05}, 
  \begin{align}
   & \int_{B(0,R)} u^2 (x,t) G_0 (x,t) \,dx
    \leq 
     C\gamma^2 
       \int_{B(0,R)} 
         \left(|x|^2 + |t| \right)^d G_0 (x,t) 
       \,dx 
  \leq C  |t|^d
   \label{EQ49}
  \end{align}
where $C$ depends on $R$. 
Combining this with \eqref{EQ04}, we get 
  \begin{equation*}
    \frac12 A_1(\delta) |t|^{m+\delta} \leq C |t|^d
   \comma t \in \left[ -\eta, 0 \right]
  \end{equation*}
which yields
  \begin{equation}
   d\leq m+ \delta
   \period
   \label{EQd_1}
  \end{equation}

For the other direction, we use 
\cite[Theorem~1.1]{AV}
on a structure of a solution in a neighborhood
of a point where $u$ vanishes.
Recall that $u$ solves
   \begin{align}
    \partial_t u - \Delta u & = a_j \partial_j u + w_j (x,t) \partial_j u + v(x,t) u  = f(x,t), \nonumber 
   \end{align}
and for all $(x,t) \in Q_1$,
   \begin{align}
    |f(x,t)|  &  \leq  \left( |a| + \Vert w\Vert_{L^{\infty} (\Omega
    \times I)} \right) \max_j |\partial_j u(x,t)| + \Vert v(x,t)\Vert_{L^{\infty} (\Omega \times I)} |u(x,t)|\nonumber \\
    & \leq C |(x,t)|^{d-1}
   \end{align} 
where $C$ depends on $u$, $v$, and $w$.
This implies that for any $r \leq 1$, and $\alpha \in (0,1)$
  \begin{align*}
   \Vert f\Vert_{L^{\infty} (Q_r)} & \leq C r^{d-1}
   \period
  \end{align*}
Therefore, by \cite[Theorem~1.1]{AV},
for the solution $u \in W^{2,1}_{\infty} (Q_1)$ of 
$(\partial_t+\Delta)u = f$, 
there exists a {\em homogeneous} caloric polynomial $P_d$ of degree 
less than or equal to $d$ 
such that for any $\alpha \in (0,1)$,
  \begin{equation}
   |u(x,t) - P_d (x,t)|
     \leq C |(x,t)|^{d+\alpha}
  \label{EQ05}
  \end{equation}
for $(x,t) \in Q_{1}$. 
Since the degree of vanishing of $u$ is $d$,
the degree of $P_d$ has to be $d$.
%
%
%
Thus we may write
  \begin{equation*}
   P_d = \sum_{|\mu|+ 2l =d} C_{\mu,l} x^{\mu} t^l
  \end{equation*}
where not all $C_{\mu,d}$ equal zero.
Then, 
by \eqref{EQ02} and~\eqref{EQ05}, 
we have for all $t \in \left[- \eta, 0 \right]$,
  \begin{align}
   & \int_{B(0,1)} 
      P_d ^2 G_0 (x,t) \,dx \nonumber 
  \\ & \indeq 
    \leq 
       C \int_{B(0,1)} u^2 G_0 (x,t) \,dx 
     + C \int_{B(0,1)} 
        \bigl( 
           |x|^2 + |t| 
        \bigr)^{d + \alpha} 
        G_0 (x,t) \,dx \nonumber 
       \\& \indeq 
         \leq 
          C |t|^{m- \delta} 
          + C \int_{\mathbb{R}^n} \left(|x|^2 + |t| \right)^{d+\alpha}  G_0 (x,t) \,dx
       \nonumber \\ & \indeq 
         = C_1 |t|^{m-\delta} + {\mathcal O}(|t|^{d+ \alpha})
   \period
 \label{EQ07}
  \end{align}
On the other hand,
   \begin{align}
    &\int_{B(0,1)}
       P_d ^2 G_0 (x,t) \,dx
        =  \sum_{|\mu| +2l =2d} 
           C_{\mu,l} 
           \int_{B(0,1)} x^{\mu} |t|^l G_0 (x,t) \,dx  
    =  C_0|t|^d + {\mathcal O}\left( |t|^{d+1} \right),
    \mbox{~~as~$t\to0-$}
    \label{EQ08}
   \end{align}
by Lemma~\ref{L05} where $C_0>0$. Consequently, we have
   \begin{equation*}
    C_0|t|^d + {\mathcal O}\left( |t|^{d+1} \right) 
     \leq 
   C_1 |t|^{m- \delta} + C|t|^{d+ \alpha},
    \hbox{~~as~} t \to 0-
   \end{equation*}
which implies 
   \begin{equation}
    d\geq m- \delta
   \period
    \label{EQd_2}
   \end{equation}
By \eqref{EQd_1} and~\eqref{EQd_2}, we get
   \begin{equation*}
    m-\delta \leq d \leq m+\delta,
   \end{equation*}
 and since this holds for all $\delta >0$, we conclude that $d=m$.
\end{proof}

\nnewpage 
\begin{proof}[Proof of Theorem~\ref{thm1}]
In Lemmas~\ref{L03} and~\ref{L06}, we have shown that 
  \begin{align*}
   \bar{Q} (\tau) 
   \leq
   C (M_1^2+M_0^{2/3})   
  \end{align*} 
for all $\tau \in [\tau_0, \infty)$, and 
\begin{align*}
\bar{Q}(\tau) \to \frac{m}{2},
\end{align*}
as $\tau \to \infty$, where $m= O_{(0,0)} (u)$. Combining these two facts,
we arrive at the desired conclusion.
\end{proof}

\startnewsection{The case ${\mathbb R}^{n}$}{sec6}
Here we present the modifications when the periodicity
assumption is removed.
Let $u$ be a solution of \eqref{EQ11} defined for
$(x,t)\in \mathbb{R}^n \times I$, where
$I$ is an open interval containing $[T_0, T_0 +T]$.
Instead of periodicity,
we assume
\eqref{EQ55}, where $C$ is a constant.
The coefficients $v$ and $w$ 
satisfy
\eqref{EQ12} and \eqref{EQ13}.
For simplicity, we assume
$M_0, M_1 \geq 1$.
Let
$q_0$ be an upper bound for the Dirichlet quotients of $u$ on
$[T_0,T_0+T]$, i.e.,
  \begin{equation}
   \frac{
    \Vert \nabla u(\cdot,t)\Vert_{L^2}^2
       }{
    \Vert u(\cdot,t)\Vert_{L^{2}}^2  
   }
   \le
   q_0
   \comma t\in [T_0,T_0+T]
   \period
   \label{EQ66}
  \end{equation}
The following is the analog of Theorem~\ref{thm1}.

\cole
\begin{Theorem}\label{thm2}
Let $u \in W^{2,1}_{\infty} (\Omega \times I)$ be a 
nontrivial solution of \eqref{EQ11} 
with $w_j$ and $v$ satisfying \eqref{EQ12} and~\eqref{EQ13},
respectively. Then the order of vanishing of $u$ at $(0, T_0+T)$
satisfies
  \begin{align}
   O_{(0,T_0+T)} (u) \leq C ( M_1 ^2 + M_0 ^{2/3} )
  \end{align}
where the constant depends on $T$, $M$, and $q_0$.
\end{Theorem}
\colb

The proof is the same as in the periodic case except that
we need to modify
Lemma~\ref{L01}, stated next.
Assume, as above, that $T_0=-T$ so that 
the time interval is
$[-T,0]$.

\cole
\begin{Lemma}
\label{L07}
Let $u$ be as above with
$v$ and $w_j$ satisfying  \eqref{EQ12} and~\eqref{EQ13}. 
Let $\epsilon\in(0,T)$ be such that
  \begin{equation}
   \epsilon\le \frac{1}{C}
   \label{EQ63}
  \end{equation}
for a sufficiently large constant $C$.
Then 
  \begin{align}
   \frac{\epsilon \int_{\mathbb{R}^n} |\nabla u(x_{\epsilon} +y, -\epsilon)|^2 G_0 (y, -\epsilon) \,dy }{\int_{\mathbb{R}^n}u(x_{\epsilon} +y, -\epsilon)^2 G_0 (y, -\epsilon) \,dy} 
     \leq  
      C M 
      \epsilon q(-\epsilon)
    \label{vo12a}
  \end{align}
for some $x_{\epsilon} \in B_2$.
\end{Lemma}
\colb

\begin{proof}[Proof of Lemma~\ref{L07}]
Assume that, for some $\lambda >0$, we have 
  \begin{align}
  \lambda 
    \int_{\mathbb{R}^n} u(x+y, -\epsilon)^2 G_0 (y, -\epsilon) \,dy 
    \leq 
    \int_{\mathbb{R}^n} |\nabla u (x+y, -\epsilon)|^2 G_0 (y, -\epsilon)\,dy 
   \comma x\in B_2
   \period
   \label{EQ58}
  \end{align}
We shall show that this cannot hold if
$\lambda \ge C M q(-\epsilon)$ for a sufficiently
large constant $C>0$.
We integrate \eqref{EQ58} in $x$ over $B_2$
and change the order of integration to obtain
  \begin{align}
   &
   \int_{{\mathbb R}^{n}}
     G_0 (y, -\epsilon) 
     \,dy
     \int_{B_2} u(x+y, -\epsilon)^2 
    \,dx
    \leq 
    \frac{1}{\lambda}
    \int_{{\mathbb R}^{n}}
     G_0 (y, -\epsilon)
     \,dy
     \int_{B_2} 
       |\nabla u (x+y, -\epsilon)|^2 
     \,dx
   \label{EQ59}
  \end{align}
Note that the right side of \eqref{EQ59} is bounded
from above
by 
  \begin{equation}
   \frac{1}{\lambda}
   \Vert \nabla u(\cdot,-\epsilon)\Vert_{L^2}^2
   \int G_0(y,-\epsilon)
   \,dy
   = \frac{(2\pi)^{n/2}}{\lambda}
   \Vert \nabla u(\cdot,-\epsilon)\Vert_{L^2}^2
   \label{EQ60}
  \end{equation}
while the left side of \eqref{EQ59} equals
  \begin{align}
   &
   \int_{{\mathbb R}^{n}}
    G_0(y,-\epsilon)\,dy
    \int_{B_2(y)}
    u(x,-\epsilon)^2 \,dx
    \nonumber\\&\indeq
    = 
   \int_{B_{1/2}}
    G_0(y,-\epsilon)\,dy
    \int_{B_2(y)}
    u(x,-\epsilon)^2 \,dx
    +
    \int_{{\mathbb R}^{n}\backslash B_{1/2}}
    G_0(y,-\epsilon)\,dy
    \int_{B_2(y)}
    u(x,-\epsilon)^2 \,dx
    \nonumber\\&\indeq
   \ge
   \int_{B_{1/2}}
    G_0(y,-\epsilon)\,dy
    \int_{B_2(y)}
    u(x,-\epsilon)^2 \,dx
   \period
   \label{EQ61}
  \end{align}
The double integral on the 
far right side is bounded from below by
  \begin{align}
   &
   \int_{B_{1/2}}
     G_{0}(y,-\epsilon)
     \,dy
     \int_{B_1(0)}
      u(x,-\epsilon)^2
     \,dx
    \nonumber\\&\indeq
    \ge
    \frac{1}{2}(2\pi)^{n/2}
    \int_{B_1(0)}
    u(x,-\epsilon)^2
    \,dx
    \ge
    \frac{1}{C M}
    \int_{{\mathbb R}^{n}}
    u(x,-\epsilon)^2
    \,dx
   \label{EQ62}
  \end{align}
where the first inequality holds
by \eqref{EQ63} if $C$ is sufficiently large
and the second one holds by \eqref{EQ55}.
Therefore, we get
  \begin{align}
   &
   \Vert u(x,-\epsilon) \Vert_{L^2}^2
   \le
   \frac{C M}{\lambda}
   \Vert \nabla u(\cdot,-\epsilon)\Vert_{L^2}
   \period
   \label{EQ65}
  \end{align}
Using
\eqref{EQ66},
we get a contradiction
if $\lambda \ge C M q(-\epsilon)$.
\end{proof}

\nnewpage
\section*{Acknowledgments} 
The authors were supported in part by the NSF grants DMS-1311943
and DMS-1615239.

\end{document}